\documentclass[12pt]{amsart}

\usepackage{mathrsfs}
\usepackage{bbm}					
\usepackage[parfill]{parskip} 
		
\usepackage{bm}	
\usepackage{amsmath,amssymb,amsthm,eucal,verbatim}
\usepackage{hyperref,amsbsy}
\usepackage{graphicx}
\usepackage{epstopdf}

\usepackage{enumitem}

\usepackage{mathtools}

\advance \topmargin by-\headheight
\advance \topmargin by-\headsep
\evensidemargin \oddsidemargin
\theoremstyle{plain}
\newtheorem{thm}{Theorem}[section]

\newtheorem{cor}{Corollary}
\newtheorem{lem}{Lemma}
\newtheorem*{thm*}{Theorem}
\newtheorem*{lem*}{Lemma}
\newtheorem*{cor*}{Corollary}

\newtheorem*{rem*}{Remark}
\newtheorem{prop}[thm]{Proposition}

\newtheorem*{defn*}{Definition}
\theoremstyle{definition}

\newtheorem*{hypchi}{Hypothesis $\mathscr{H}_{\alpha, \chi}$}
\newtheorem*{hypdis}{Hypothesis $\mathscr{D}$}

\newcommand{\be}{\begin{equation}}
\newcommand{\ee}{\end{equation}}
\renewcommand{\a}{\alpha}
\newcommand{\s}{\sigma}
\renewcommand{\b}{\beta}
\renewcommand{\r}{\rho }
\newcommand{\g}{\gamma }

\renewcommand{\d}{\delta}   

\newcommand{\sgn}{\operatorname{sgn}}
\newcommand{\var}{\operatorname{var}}

\newcommand*{\bigs}[1]{\scalebox{1.2}{\ensuremath#1}}

\makeatletter
\newcommand*{\bigss}[1]{{\hbox{$\left#1\vbox to11\p@{}\right.\n@space$}}}
\makeatother

\newcommand{\sdfrac}[2]{\mbox{\small$\displaystyle\frac{#1}{#2}$}}

\title[A Central Limit Theorem for Linear Combinations]
{A Central Limit Theorem  for Linear Combinations  \\ of Logarithms of Dirichlet $L$-functions}

\author{Fatma \c{C}\.{i}\c{c}ek}

\begin{document}

\maketitle

\begin{abstract}
The purpose of this paper is to generalize our earlier work on the logarithm of the Riemann zeta-function to linear combinations of logarithms of primitive Dirichlet $L$-functions with constant real coefficients. Under the assumption of suitable hypotheses, we prove that as $T\to \infty $, the sequence $(a_1\log|L(\rho,\chi_n)|+\dots+a_1\log|L(\rho,\chi_n)|)$ has an approximate Gaussian distribution with mean $0$ and variance $ \tfrac12 \big({a_1}^2+\dots+{a_n}^2\big)\log\log T.$ Here $a_1, \dots, a_n \in \mathbb{R}$, $0< \Im\rho \leq T$ where $\rho$ runs over nontrivial zeros of the zeta-function, and each of the $\chi_j$ is a primitive Dirichlet character with conductor $M_j$ for $M_j\leq T.$ From the proof of this result, we also derive the independence of the distributions of sequences $(\log|L(\rho,\chi_1)|), \dots,$ $(\log|L(\rho,\chi_n)|)$ provided that they are suitably normalized.
\end{abstract}

\section{Introduction}\label{intro}

Let $L(s, \chi)$ be a Dirichlet $L$-function corresponding to a primitive Dirichlet character. The Generalized Riemann hypothesis (GRH) claims that all nontrivial zeros of Dirichlet $L$-functions lie on the critical line. This implies that both a Dirichlet $L$-function $L(s, \chi)$ and the Riemann zeta-function $\zeta(s)$ have their nontrivial zeros on the same line, that is, $\Re s=\tfrac12.$ It is further conjectured that such zeros never coincide.

\begin{hypdis}
The nontrivial zeros of primitive Dirichlet $L$-functions never coincide with the zeros of the Riemann zeta-function. 
\end{hypdis}
Note that this hypothesis is an implication of the well-known Grand Simplicity Hypothesis~\cite{RS}. 

On the assumption of the GRH and Hypothesis $\mathscr{D}$, it is of interest to study the distribution of the sequence $(\log|L(\r, \chi)|)$, where $\displaystyle \r=\tfrac12+i\g$ runs over the nontrivial zeros of the zeta-function. Due to its similarity, we will consider a more general case and study the distribution of sequences that are given as
    \[
    \mathcal{L}(\r)=a_1 \log|L(\r, \chi_1)|+ \dots+a_n \log|L(\r, \chi_n)|, \quad \text{where} \quad a_1, \dots, a_n \in \mathbb R. 
    \]
As a leading mark to this investigation, it was previously shown by the author~\cite{log zeta} that for $0<\g \leq T$ and $z \neq 0$ sufficiently small, the sequence $(\log|\zeta(\rho+z)|)$ has an approximate Gaussian distribution with a mean that depends on $\rho$ and $z$ and variance $\tfrac12 \log\log T$ as $T\to \infty$. This assumed the Riemann hypothesis and the famous pair correlation conjecture of Montgomery~\cite{Montgomery}. As we will see, in addition to the GRH, a similar result for the sequence $(\mathcal{L}(\r))$ requires a zero-spacing hypothesis that is about the vertical distribution of zeros of the $L$-function $L(s, \chi_j)$ relative to the vertical distribution of zeros $\rho$ of $\zeta(s)$ for each $1\leq j \leq n$.

Let $0<\a \leq 1$ be fixed, and $\r_{\chi}=\tfrac12+i\g_\chi$ denote a generic nontrivial zero of $L(s, \chi)$. Also, let $N(T)$ denote the number of zeros $\rho=\frac12+i\g$ of the Riemann zeta-function with $0<\g \leq T$.

\begin{hypchi}\label{hypothesis chi}
Uniformly for  $0<C<1$, we have
        \[
        \limsup_{T\to\infty}\frac{1}{N(T)}
        \#\bigg\{0 < \g \leq T: \text{There exists a }\,  \r_\chi \text{ such that } \bigs|\g_\chi-\g\bigs| \leq \frac{C}{\log T}\bigg\} \ll C^\alpha.
        \]
\end{hypchi}

This is analogous to Hypothesis $\mathscr{H}_\a$, which was introduced in~\cite{H} to study the distribution of the sequence $(\log|\zeta'(\r)|)$ as $0 < \Im \rho \leq T$ on the assumption of a zero-spacing hypothesis that is weaker than the pair correlation conjecture. 

Our first result is the following.

\begin{thm}\label{distr of linear combination}
Let $T$ be sufficiently large. Set
\[
\mathcal{L}(\r)=a_1\log|L(\r, \chi_1)|+\dots+a_n\log|L(\r, \chi_n)|, \quad \text{where} \quad a_1, a_2, \dots, a_n \in \mathbb R,
\]
and $\chi_1,\dots, \chi_n$ are distinct primitive Dirichlet characters with conductors bounded by $T.$ Assume the Generalized Riemann hypothesis and Hypothesis $\mathscr{D}.$ Further, suppose that for each $1\leq j \leq n$, Hypothesis $\mathscr{H}_{\a,\chi_j}$ is true for some $\a\in(0,1].$ For $A<B$, we have
\begin{align*}
\frac{1}{N(T)}\#\bigg\{0< \g \leq T:  \frac{\mathcal{L}(\rho)}{\sqrt{\left(\tfrac{1}{2}\big( {a_1}^2+\dots+{a_n}^2\big)\right) \log\log T}} &\in [A, B]\bigg\} \\
=\frac{1}{\sqrt{2\pi}} &\int_A^B e^{-\tfrac{x^2}{2}} \mathop{dx}  
+O\bigg(\sdfrac{(\log\log\log T)^2}{\sqrt{\log\log T}}\bigg).
\end{align*}
\end{thm} 

On the assumption of the GRH and Hypothesis $\mathscr{D}$, and thus without assuming a zero-spacing hypothesis, one can prove a similar theorem for sequences of the form 
\[
(a_1\arg L(\r, \chi_1)+\dots+a_n\arg L(\r, \chi_n)),\quad \text{where} \quad a_1, \dots, a_n \in \mathbb{R},
\]
with a slightly better error term (see Theorem 1.2 in~\cite{log zeta}). 

Moreover, from the discussion of the proof of the above theorem and by using some other ideas, the following theorem will follow.

\begin{thm}\label{independence}
Under the same conditions and with the same notation as in the above theorem, the sequences
         \[
        \bigg(\sdfrac{\log|L(\r, \chi_1)|}{\sqrt {\tfrac12 \log\log T}}\bigg), \dots,\bigg(\sdfrac{\log|L(\r, \chi_n)|}{\sqrt {\tfrac12 \log\log T}}\bigg)
         \]
 are approximately independent as $T\to \infty$. That is, for real numbers $A_1, \dots, A_n$, $B_1, \dots, B_n$ such that $A_j < B_j$ for each $1\leq j\leq n$, we have
        \begin{align*}
        \lim_{T\to\infty } \frac{1}{N(T)}
        \#\bigg\{0<\g \leq T&: \sdfrac{\log|L(\r, \chi_j)|}{\sqrt {\tfrac12 \log\log T}} \in [A_j, B_j] \,\, \text{ for all } \, 1\leq j \leq n \bigg\}  
        \\
        =& \prod_{j=1}^n \lim_{T\to\infty }\frac{1}{N(T)}\#\bigg\{0<\g \leq T: 
        \sdfrac{\log|L(\r, \chi_j)|}{\sqrt {\tfrac12 \log\log T}} \in [A_j, B_j] \bigg\}.
        \end{align*}
\end{thm}

According to a result in probability~\cite[p. 120]{GW}, the distribution of a linear combination of independent Gaussian random variables with real coefficients is also Gaussian. Therefore, in view of Theorem \ref{distr of linear combination}, the above theorem is not unexpected.

Further, we remark that by using the same ideas as in the proof of Theorem \ref{independence}, one can also prove that the sequences
         \[
        \bigg(\sdfrac{\log|L(\r, \chi_1)|}{\sqrt {\tfrac12 \log\log T}}\bigg), \bigg(\sdfrac{\arg L(\r, \chi_1)}{\sqrt {\tfrac12 \log\log T}}\bigg)
         \dots,\bigg(\sdfrac{\log|L(\r, \chi_n)|}{\sqrt {\tfrac12 \log\log T}}\bigg), \bigg(\sdfrac{\arg L(\r, \chi_n) }{\sqrt {\tfrac12 \log\log T}}\bigg)
         \]
are approximately independent. 

As an interesting application of Theorem \ref{distr of linear combination}, one can study the proportion of nontrivial zeros $\rho$ of the Riemann zeta-function such that $F(\r)=0$ for
    \[
    F(s)= b_1 L(s, \chi_1)+\dots +b_n L(s, \chi_n),
    \]
where the $b_j$ can be complex numbers. A seminal work about zeros of such functions is that of Bombieri and Hejhal~\cite{BH}. Under suitable hypotheses, they proved that almost all zeros of $F(s)$ lie on the critical line if the $b_j$ are real and if the $L(s, \chi_j)$ satisfy the same functional equation. Their result can be viewed as an application of Selberg's central limit theorem, which can be generalized to state that for $0<t<T$, the function
    \[
    a_1\log|L(\tfrac12+it,\chi_1)|+\dots+a_n\log|L(\tfrac12+it,\chi_n)|,
    \]
where $a_1, \dots, a_n \in \mathbb R$ and $\chi_1, \dots, \chi_n$ are distinct primitive Dirichlet characters, has an approximate Gaussian distribution with mean $0$ and variance $\left(\tfrac12 \sum_{j=1}^n {a_j}^2\right) \log\log T$ as $T\to \infty$. Since Theorem \ref{distr of linear combination} is reminiscent of this version of Selberg's central limit theorem, we can hope to prove that a positive proportion of nontrivial zeros of the zeta-function do not happen to be zeros of $F(s).$ This direction of research will be pursued in an upcoming work.

Our goal is to prove Theorem \ref{distr of linear combination} for $n=2$ and Theorem \ref{independence} for any integer $n\geq 2$. A proof of Theorem \ref{distr of linear combination} for a general integer $n$ would proceed similarly.

Throughout, we assume that both the GRH and Hypothesis $\mathscr{D}$ hold. We take $T$ to be a sufficiently large positive real number. The letters $c$ and $D$ will stand for positive constants that may be different at each occurrence. $c$ stands for an absolute constant, but $D$ may depend on some parameters. The variables $p$ and $q,$ indexed or not, always denote prime numbers, and the variables $j, k, \ell, m$ and $n$ always denote nonnegative integers.

%%%%%%%%%%%%%%%% %%%%%%%%%%%%%%%% 

\section*{Acknowledgements}
The author thanks Steven M. Gonek for his guidance and support during the study of this problem.

%%%%%%%%%%%%%%%% %%%%%%%%%%%%%%%% 

\section{Preliminaries}\label{sec:preliminaries}

Let $L(s, \chi)$ be the Dirichlet $L$-function corresponding to a primitive Dirichlet character $\chi$ modulo $M$. We put $\mathfrak{a}=0$ if $\chi(-1)=1$, and $\mathfrak{a}=1$ if $\chi(-1)=-1$. Then $L(s, \chi)$ has a trivial zero at $s=-2n-\mathfrak{a}$ for each positive integer $n$.  

Corresponding to such a Dirichlet character, we define 
   \[
   \CMcal{P}_{\chi}(\g):=\sum_{p\leq X^2}\frac{\chi(p)}{p^{1/2+i\g}}.
    \]
The length of this Dirichlet polynomial is denoted by $X^2$ due to our use of the weight function
       \[
       \Lambda_X(n):=\Lambda(n)w_X(n),
        \]
 where    
        \[
        \begin{split}
        w_X(n):=
        \begin{cases}
        1&\quad \text{if} \quad 1\leq n \leq X,\\
         \frac{\log{(X^2/n)}}{\log{X}} &\quad \text{if} \quad  X< n \leq X^2 .
        \end{cases}
        \end{split}
        \]
$X$ will later be chosen suitably. As our first result, we will prove that the real part of $\CMcal{P}_{\chi}(\g)$ approximates $\log|L(\r, \chi)|$. Define
    \[
    \eta_{\chi}(\g):=\min_{\g_\chi} |\g-\g_\chi|,
    \]
where $\g_\chi$ runs over the ordinates of nontrivial zeros $\r_\chi=\frac12+i\g_\chi$ of $L(s, \chi)$. Note that $ \eta_\chi(\g)$ never vanishes due to our assumption of Hypothesis $\mathscr{D}$.

\begin{lem}\label{log L approximate formula}
    Assume that both the GRH and Hypothesis $\mathscr{D}$ hold. Let $4\leq X\leq T^2$. Then 
        \[
        \log|L(\r, \chi)|
        =\Re\CMcal{P}_{\chi}(\g)
        +O\bigg(\sum_{j=1}^{4}r_{\chi}^j(X, \g)\bigg)
        +O\Big(\sdfrac{\log(M|\g|)}{\log X}\Big)+O(1),
        \]
where
        \be\notag
        \begin{split}
        r_{\chi}^1(X, \g)&=\Big|\sum_{p\leq X^2}\sdfrac{(1-w_X(p))\chi(p)}{p^{1/2+i\g}}\Big|, \qquad
        r_{\chi}^2(X, \g)=\Big|\sum_{p\leq X}\sdfrac{w_X(p^2)\chi(p^2)}{p^{1+2i\g}}\Big|, \\
        &r_{\chi}^3(X, \g)=\frac{1}{\log X} \int_{1/2}^{\infty} X^{\frac 1 2-\s}\Big|\sum_{p\leq X^2}\sdfrac{\Lambda_X(p)\chi(p) \log{(Xp)}}{p^{\s+i\g}}\Big|\mathop{d\s}, \\ \qquad
        &\qquad r_{\chi}^4(X, \g)=\bigg(1+\log^+\frac{1}{\eta_{\chi}(\g)\log X}\bigg)\sdfrac{E_\chi(X, \g)}{\log X},
        \end{split}
        \ee
for
        \[
        E_\chi(X, \g) =\Big|\sum_{n\leq X^2} \sdfrac{\Lambda_X(n)\chi(n)}{n^{\s_1+i\g}}\Big|+\log(M|\g|). 
        \]
\end{lem}

\begin{proof}
    Let $4\leq X \leq t^2$ and $\s_1=\frac12+\frac{4}{\log X}$. We write $s=\s+it$. 
For $s\neq \r_\chi$ or $(-2n-\mathfrak{a})$ for any positive integer $n$, we have
        \be\label{eq:L' by L}
        \begin{split}
        &\frac{L'}{L}(s, \chi)
        =-\sum_{n\leq X^2}\frac{\Lambda_X(n)\chi(n)}{n^{s}}  \\
        &+\frac{1}{\log X}\sum_{\r_\chi} \frac{X^{\r_\chi-s}(1-X^{\r_\chi-s})}{(\r_\chi-s)^2}
        +\frac{1}{\log X}\sum_{n=1}^\infty \frac{X^{-2n-\mathfrak{a}-s}(1-X^{-2n-\mathfrak{a}-s})}{(2n+\mathfrak{a}+s)^2}.
        \end{split}
        \ee
The proof of this is standard (see Lemma 15 in \cite{SelbergLFncs}). For now, suppose that $\s\geq \s_1$. Then the third term on the right-hand side of \eqref{eq:L' by L} is clearly $\displaystyle \ll \tfrac{1}{\sqrt{X}\log X}$. The second term on the right-hand side is
        \[
        \ll  \sum_{\g_\chi}\frac{\s_1-1/2}{\big(\s_1-1/2\big)^2+(t-\g_\chi)^2}. 
        \]
An important step is to prove that this sum over nontrivial zeros can be estimated in terms of a sum over prime powers. Indeed, it immediately follows from (5.3) and (5.8) of~\cite{SelbergLFncs} that
        \be\label{sum over gamma chi}
        \sum_{\g_{\chi}}\frac{\s_1-1/2}{\big(\s_1-1/2\big)^2+(t-\g_{\chi})^2}
        = O\bigs(E_{\chi}(X, t)\bigs),
        \ee
where 
        \be\notag
        E_{\chi}(X, t)=\Big|\sum_{n\leq X^2}\sdfrac{\Lambda_X(n)\chi(n)}{n^{\s_1+it}}\Big|+\log\big(M(1+|t|)\big).
        \ee
From this and \eqref{eq:L' by L}, we then obtain
        \be \label{eq: L' L}
         \frac{L'}{L}(s,\chi)
        =- \sum_{n\leq X^2}\frac{\Lambda_X(n)\chi(n)}{n^s} 
        +O\bigs(E_{\chi}(X, t)\bigs) \quad \text{for}  \quad \s\geq \s_1.
        \ee
Also, we note the following estimate which holds for any $s$ other than $\r_\chi$~\cite[p. 83]{Davenport}.
        \be\label{L' by L over zero sum}
        \frac{L'}{L}(s, \chi)=\sum_{\r_\chi}\bigg(\frac{1}{s-\r_\chi}-\frac{1}{\r_\chi}\bigg)+O\bigs(\log\big(M(1+|t|)\big)\bigs).
        \ee 
Clearly, we can write
        \be\label{integral L' over L}
        \begin{split}
        \log |L(\r,\chi)|
        &=-\Re\int_{1/2}^{\infty} \frac{L'}{L}(\s+i\g,\chi)\mathop{d\s}  \\
        &=-\Re\int_{\s_1}^{\infty} \frac{L'}{L}(\s+i\g,\chi)\mathop{d\s} 
        -\Big(\s_1-\frac12\Big)\Re \frac{L'}{L}(\s_1+i\g,\chi)  \\ 
        &\qquad \quad 
        +\Re\int_{1/2}^{\s_1} \Big( \frac{L'}{L}(\s_1+i\g,\chi)- \frac{L'}{L}(\s+i\g,\chi) \Big)\mathop{d\s}
         \\ &=J_1+J_2+J_3.
        \end{split}
        \ee
For $J_1$, integrating \eqref{eq: L' L} over $[\s_1,\infty)$ gives
        \[
        J_1
        =\Re \sum_{n\leq X^2} \frac{\Lambda_X(n)\chi(n)}{n^{\s_1+i\g}\log n}
        +O\bigg(\frac{E_{\chi}(X, \g)}{\log X}\bigg).
        \]
Moreover, substituting $\s=\s_1$ in \eqref{eq: L' L} gives
        \[
        J_2
        \ll \frac{1}{\log X}\bigg|\frac{L'}{L}(\s_1+i\g, \chi)\bigg|
        \ll \frac{E_{\chi}(X, \g)}{\log X}.
        \]
We combine these estimates in \eqref{integral L' over L} to have
       \[
       \log|L(\r,\chi)|
       = \Re \sum_{n\leq X^2} \frac{\Lambda_X(n)\chi(n)}{n^{\s_1+i\g}\log n}
       +\Re{J_3}
        +O\bigg(\frac{E_{\chi}(X, \g)}{\log X}\bigg).
       \]
It remains to consider $\Re{J_3}$. By \eqref{L' by L over zero sum}, 
        \be\notag
        \begin{split}
        \Re&\bigg\{\frac{L'}{L}(\s_1+i\g, \chi)-\frac{L'}{L}(\s+i\g, \chi)\bigg\}\\
        &=\sum_{\g_{\chi}} \frac{\s_1-1/2}{\big(\s_1-1/2\big)^2+(\g-\g_{\chi})^2}
        -\sum_{\g_{\chi}} \frac{\s-1/2}{\big(\s-1/2\big)^2+(\g-\g_{\chi})^2} 
        +O\big(\log(M|\g|\big).
        \end{split}
        \ee        
Then, in a similar way to the proof of Lemma 2.1 in ~\cite{log zeta}, one finds that
        \be\notag
        \begin{split}
        \Re \,J_3
         \ll
        \bigg(1  + \log^{+}\Big(\frac{1}{\eta_{\chi}(\g) \log X}\Big)\bigg)  \sum_{\g_\chi}   \frac{(\s_1-1/2)^2}{(\s_1-1/2)^2 +(\g-\g_\chi)^2 }
         +O\bigg(\frac{\log(M|\g|)}{\log X}\bigg).
        \end{split}
        \ee
 Also by \eqref{sum over gamma chi}, the sum in the above is $\ll \frac{E_\chi(X, \g)}{\log X}$. Thus by \eqref{integral L' over L}, we showed that
       \[
        \log|L(\r,\chi)|
        = \Re \sum_{n\leq X^2} \sdfrac{\Lambda_X(n)\chi(n)}{n^{\s_1+i\g}\log n} 
        +O\bigg(\Big(1+\log^+{\frac{1}{\eta_{\chi}(\g)\log X}}\Big)\frac{E_\chi(X, \g)}{\log X}\bigg).
        \]
By using the argument in \cite[p. 35]{S1946Archiv}, one can easily see that the Dirichlet polynomial above is equal to
        \begin{align*} 
         \Re \CMcal{P}_{\chi}(\g)
        &+O\Big(\Big|\sum_{p\leq X^2}\sdfrac{(1-w_X(p))\chi(p)}{p^{1/2+i\g}}\Big|\Big)
        +O\Big(\Big|\sum_{p\leq X}\sdfrac{w_X(p^2)\chi(p^2)}{p^{1+2i\g}}\Big|\Big)\\
        &+O\bigg( \frac{1}{\log X} \int_{1/2}^\infty X^{\frac 1 2-\s}\Big|\sum_{p\leq X^2}\sdfrac{\Lambda_X(p)\chi(p)\log{(Xp)}}{p^{\s+i\g}}\Big|\mathop{d\s}\bigg)+O(1). 
        \end{align*}
From these two statements, the result follows.
   \end{proof}    

Now, let $\chi_1$ and $\chi_2$ be distinct primitive Dirichlet characters with respective conductors $M_1$ and $M_2$. As our notation, we set
     \[
     \mathcal{L}(\rho) := a_1\log|L(\r, \chi_1)|+a_2\log|L(\r, \chi_2)|
     \]
for fixed real numbers $a_1$ and $a_2$. By the above lemma, we conclude that the real part of the Dirichlet polynomial
    \[
    \CMcal{P}_{\mathcal{L}}(\g):= \sum_{p\leq X^2} \frac{a_1 \chi_1(p)+a_2 \chi_2(p)}{p^{1/2+i\g}}
    =a_1 \CMcal{P}_{\chi_1}(\g)+a_2 \CMcal{P}_{\chi_2}(\g)
    \]
approximates $ \mathcal{L}(\rho)$. Indeed, on average, the moments of the distance between them can be bounded as below.
  
     \begin{prop}\label{linear combination error}
Let $k$ be a positive integer, and $\displaystyle T^{\frac{\d}{16k}} \leq X \leq T^{\frac{1}{16k}}$ for a fixed number $0< \d <1$. Further, let $\chi_1, \chi_2$ be primitive Dirichlet characters modulo $M_1, M_2$, respectively, and $M_1, M_2 \leq T$. Assume that both the GRH and Hypothesis $\mathscr{D}$ hold. Also, suppose that both Hypothesis $\mathscr{H}_{\a,{\chi}_1}$ and Hypothesis $\mathscr{H}_{\a,{\chi}_2}$ are true for some $\a\in(0,1]$. Then
    \[
    \sum_{0 < \g \leq T} 
    \bigs( \mathcal{L}(\rho) - \Re\CMcal{P}_{\mathcal{L}}(\g) \bigs)^{2k}
    =O\bigs((Dk)^{4k}N(T)\bigs)
    \]
    for a constant $D$ depending on $\a$ and $\d.$
    \end{prop}

\begin{proof}
Note the well-known inequality
           \[
         \Big(\frac 1r \sum_{j=1}^r x_j \Big)^k \leq \frac{1}{r} \sum_{j=1}^r x_j^k \quad 
       \text{for} \quad x_1, x_2, \dots, x_r>0 \quad \text{and} \quad k\geq 1.
         \]     
By using this, we can write   
    \begin{align*}
     &\sum_{0 < \g \leq T} \bigs(\mathcal{L}(\rho) - \Re \CMcal{P}_{\mathcal{L}}(\g) \bigs)^{2k} \\
     &\ll  
     c^{2k}\sum_{0 < \g \leq T}\bigs( \log|L(\r, \chi_1)|- \Re\CMcal{P}_{\chi_1}(\g) \bigs) ^{2k}
      + c^{2k} \sum_{0 < \g \leq T}\bigs( \log|L(\r, \chi_2)|- \Re{\CMcal{P}_{\chi_2}(\g)} \bigs) ^{2k}. 
    \end{align*}  
From the argument in the proof of Proposition 4.2 in~\cite{log zeta}, it follows that both of the sums on the last line are  $\ll (Dk)^{4k}N(T)$ for a constant $D=D(\a, \d)$. This uses the fact that $|\chi(n)| \leq 1$, and also that the quantity $\eta_{\chi}(\g)$ is similarly defined as the quantity $\eta_\g=\min_{\g' \neq \g} |\g'-\g|$ in~\cite{log zeta}. We also use the bound  
$
\log^{2k}(M_j\g)/(\log X)^{2k} \ll (ck)^{2k} N(T),
$
which follows from our assumption that $M_j \leq T$.
\end{proof}

For our later use, we note the following lemma about the moments of $|\Re\CMcal{P}_{\mathcal{L}}(\g)|$.

\begin{lem}\label{our lemma 3.16 linear combination}
    Let $\displaystyle X\leq T^{\frac{1}{16k}}$, and define $\Psi= \sum_{p\leq X^2} p^{-1}$. We have
    \[
    \sum_{0 < \g \leq T}\big|\Re{\CMcal{P}_{\mathcal{L}}(\g)}\big|^k= O\bigs((ck\Psi)^{k/2}N(T)\bigs). 
    \]
\end{lem}

\begin{proof}
By Lemma 3.2 in~\cite{log zeta} and Stirling's formula, 
    \[
    \sum_{0 < \g \leq T}\big|\Re{\CMcal{P}_{\chi_j}(\g)}\big|^k= O\bigs((ck\Psi)^{k/2}N(T)\bigs) \quad \text{for} \quad j=1, 2.  
    \]
The result then follows easily.
\end{proof}

Another object that we introduce is a random polynomial that approximates the behavior of the Dirichlet polynomial $\CMcal{P}_{\mathcal{L}}(\g)$. Let $(\theta_p )_{p:\text{prime}}$ be a sequence of identically distributed independent random variables each of which is uniformly distributed over $[0, 1]$. It is a well-known idea that the random polynomial
    \[
    \CMcal{P}_{\chi}(\underline{\theta})= \sum_{p\leq X^2}\frac{\chi(p)e^{2\pi i \theta_{p}}}{\sqrt p}.
    \]
can be used to model the distribution of $\CMcal{P}_{\chi}(\g)$. Here, $\Re\CMcal{P}_{\chi}(\underline{\theta})$ has mean $0$ and variance $\frac12 \sum_{p\leq X^2}\frac 1p$. Note that the sum in the latter equals $\log\log X+O(1)$, which is of size $\log\log T$ for suitable $X$. We thus define the quantities 
    \[
    \Psi= \sum_{p\leq X^2} \frac 1p  \quad \text{and} \quad
    \Psi(T)= \sum_{p\leq T} \frac 1p
    \]
to use later. Let an integer $n>1$ have the unique factorization $\displaystyle n=p_1^{\mu_1}\dots p_r^{\mu_r}.$ We set
    \[
    \theta_n:=\mu_1\theta_{p_1}+\dots+\mu_r\theta_{p_r}.
    \]    
It is easy to see that
    \[
    \theta_{mn}=\theta_m+\theta_n \quad  \text{for any } \mkern9mu m, n \in \mathbb{Z}^+ \quad \text{and} \quad \theta_m=\theta_n \quad \text{if and only if} \quad m=n.
    \]    
A useful consequence of this is 
    \be\label{orthogonality}
    \int_0^1 e^{2\pi i\theta_m} e^{-2\pi i\theta_n} \mathop{d\underline{\theta}} = 
    \begin{cases}
    1 &\text{if} \quad m=n, \\
    0 &\text{if} \quad m\neq n,
    \end{cases}
    \ee
where and also from now on, $\int_0^1 (\dots) \mathop{d\underline{\theta}}$ denotes a multidimensional integral.
    
The real part of the following random polynomial mimics the behavior of $\Re\CMcal{P}_{\mathcal{L}}(\g)$. 
\[
\CMcal{P}_{\mathcal{L}}(\underline{\theta}) := \sum_{p\leq X^2}\frac{(a_1\chi_1(p)+a_2\chi_2(p))e^{2\pi i \theta_{p}}}{\sqrt p}
= a_1\CMcal{P}_{\chi_1}(\underline{\theta})+a_2\CMcal{P}_{\chi_2}(\underline{\theta}).
\]
We easily compute its mean.
    \[
    \int_0^1 \Re\CMcal{P}_{\mathcal{L}}(\underline{\theta}) \mathop{d\underline{\theta}} 
    =a_1  \int_0^1 \Big(\Re \sum_{p\leq X^2} \frac{\chi_1(p)e^{2\pi i\theta_p}}{\sqrt p}\Big) \mathop{d\underline{\theta}}
    +a_2 \int_0^1 \Big(\Re \sum_{p\leq X^2} \frac{\chi_2(p)e^{2\pi i\theta_p}}{\sqrt p}\Big) \mathop{d\underline{\theta}}
    =0.
    \]
Its variance equals
    \be\label{compute variance linear combination}
    \begin{split}
    \int_0^1\big(\Re\CMcal{P}_{\mathcal{L}}(\underline{\theta})\big)^2 \mathop{d\underline{\theta}} 
    =&\, \, {a_1}^2 
    \int_0^1 \big(\Re\CMcal{P}_{\chi_1}(\underline{\theta})\bigs)^2 
    \mathop{d\underline{\theta}}
    +{a_2}^2 \int_0^1 \big(\Re\CMcal{P}_{\chi_2}(\underline{\theta})\bigs)^2 
    \mathop{d\underline{\theta}} \\
    &+ 2a_1 a_2 \int_0^1 \Re\CMcal{P}_{\chi_1}(\underline{\theta}) 
    \Re\CMcal{P}_{\chi_2}(\underline{\theta})  \mathop{d\underline{\theta}}. 
    \end{split}
    \ee
We already know that
  \[
   \int_0^1 \big(\Re\CMcal{P}_{\chi_1}(\underline{\theta})\bigs)^2 
    \mathop{d\underline{\theta}}
    = \frac{\Psi}{2}
    = \int_0^1 \big(\Re\CMcal{P}_{\chi_2}(\underline{\theta})\bigs)^2 
    \mathop{d\underline{\theta}}.
  \]
The last term on the right-hand side of \eqref{compute variance linear combination} can be seen to be smaller. Easily,
    \begin{align*}
    \int_0^1\bigs(\Re\CMcal{P}_{\chi_1}(\underline{\theta}) 
    \Re\CMcal{P}_{\chi_2}(\underline{\theta})\bigs)    \mathop{d\underline{\theta}}  
    &= \frac14 \int_0^1 \bigs(\CMcal{P}_{\chi_1}(\underline{\theta}) +
    \overline{\CMcal{P}_{\chi_1}(\underline{\theta})}\bigs)
     \bigs(\CMcal{P}_{\chi_2}(\underline{\theta}) +
    \overline{\CMcal{P}_{\chi_2}(\underline{\theta})}\bigs) 
    \mathop{d\underline{\theta}} \\
    &=  \frac12 \int_0^1 \Re\bigs(\CMcal{P}_{\chi_1}(\underline{\theta}) \CMcal{P}_{\chi_2}(\underline{\theta})\bigs)
     \mathop{d\underline{\theta}}
    +  \frac12 \int_0^1 \Re\bigs(\CMcal{P}_{\chi_1}(\underline{\theta})\overline{\CMcal{P}_{\chi_2}(\underline{\theta})} \, \bigs) \mathop{d\underline{\theta}}. 
    \end{align*}
For the terms on the last line, again by \eqref{orthogonality} we have
    \[
    \Re\int_0^1\CMcal{P}_{\chi_1}(\underline{\theta}) \CMcal{P}_{\chi_2}(\underline{\theta})\mathop{d\underline{\theta}}
      =0,
    \]
and
    \[
    \Re \int_0^1\CMcal{P}_{\chi_1}(\underline{\theta})\overline{\CMcal{P}_{\chi_2}(\underline{\theta})}  \mathop{d\underline{\theta}}
    = \Re\sum_{p\leq X^2} \frac{\chi_1(p) \overline{\chi}_2(p)}{p}.
    \]
Here $\chi_1\overline{\chi_2}$ is a nonprincipal Dirichlet character modulo $\operatorname{lcm}[M_1, M_2]$, and we have
     \[
     \sum_{p\leq X^2} \frac{\chi_1(p)\overline{\chi_2}(p)}{p}
     =\log {L(1, \chi_1\overline{\chi_2})}+O_M(1)
     =O_M(1). 
     \]
Thus by \eqref{compute variance linear combination},
    \be\label{variance linear combination}
      \int_0^1\big(\Re\CMcal{P}_{\mathcal{L}}(\underline{\theta})\big)^2 \mathop{d\underline{\theta}}
    =\frac{{a_1}^2+{a_2}^2}{2}\Psi+O_M(1).
    \ee
Note that for suitable $X$, this is close to $\frac{{a_1}^2+{a_2}^2}{2} \log\log T$, that is, the variance of the approximate Gaussian distribution in the statement of Theorem \ref{distr of linear combination}. 

As a similar result to Lemma \ref{our lemma 3.16 linear combination}, we have 

\begin{lem}\label{Tsang lemma 3.4 linear combination}
    Suppose that $k$ is a nonnegative integer. For $\displaystyle 2\leq X \leq T^{\frac{1}{2k}}$, we have
    \be\notag
    \int_0^1\big|\Re{\CMcal{P}_{\mathcal{L}}(\underline{\theta})}\big|^k\mathop{d{\underline{\theta}}}
    = O\bigs((ck\Psi)^{k/2}\bigs).
    \ee
\end{lem}

\begin{proof}
As in (3.10) in \cite{Tsang}, one can use \eqref{orthogonality} and easily prove that 
    \be\notag
    \int_0^1\big|\Re{\CMcal{P}_{\chi_j}(\underline{\theta})}\big|^k\mathop{d{\underline{\theta}}}
    = O\bigs((ck\Psi)^{k/2}\bigs) \quad \text{for} \quad j=1, 2. 
    \ee
The claim of the lemma is then straightforward.    
\end{proof}

\section{Proof of Theorem \ref{distr of linear combination}} 

The proof is based on a computation of the Fourier transform of $\Re\CMcal{P}_{\mathcal{L}}(\g)$ in terms of the Fourier transform of $\Re{\CMcal{P}_{\mathcal{L}}(\underline{\theta})}$. As we will see, the latter can be expressed in terms of values of the Bessel functions of the first kind and thus can be computed explicitly. 

As the first step, we express the $k$th moment of $\Re\CMcal{P}_{\mathcal{L}}(\g)$ in terms of the $k$th moment of $\Re{\CMcal{P}_{\mathcal{L}}(\underline{\theta})}$. 
  
\begin{lem}\label{lemma 3.4 linear combination}
    Let $k$ be a positive integer with $k \ll \sqrt{\log T},$ and choose $\displaystyle X\leq T^{\frac{1}{16k}}$. Then
    \be\label{eq:lemma 3.4 linear combination}
    \begin{split}
    \sum_{0 < \g  \leq T} &\bigs(\Re\CMcal{P}_{\mathcal{L}}(\g)\bigs)^k 
    = N(T)\int_0^1 \bigs(\Re\CMcal{P}_{\mathcal{L}}(\underline{\theta})\bigs)^k\mathop{d{\underline{\theta}}}  \\
    &-\frac{T}{\pi}\sum_{q\leq X^2} \sum_{1\leq \ell\leq k} \frac{\log q}{q^{\ell/2}}
    \int_0^1\bigs(\Re\CMcal{P}_{\mathcal{L}}(\underline{\theta})\bigs)^k 
    \Re(e^{2\pi i\ell\theta_q}) \mathop{d{\underline{\theta}}} 
    +O\bigs( (ck)^k\sqrt T\log^3{T}\bigs).
    \end{split}
    \ee
\end{lem}

\begin{proof}
By the binomial theorem,
        \[
        \sum_{0 < \g \leq T}\bigs(\Re \CMcal{P}_{\mathcal{L}}(\g)\bigs)^k
        =\frac{1}{2^k}\sum_{j=0}^k \binom{k}{j}
        \sum_{0 < \g \leq T}  \CMcal{P}_{\mathcal{L}}(\g)^j 
        \overline{\CMcal{P}_{\mathcal{L}}(\g)}^{k-j}.
        \]
For $1\leq j\leq k$, 
        \be\label{j term linear combination}
        \sum_{0 < \g \leq T}\CMcal{P}_{\mathcal L}(\g)^j 
         \overline{\CMcal{P}_{\mathcal L}(\g)}^{k-j}  
        =\sum_{0 < \g \leq T}\Big(\sum_{p\leq X^2}\sdfrac{a_1\chi_1(p)+a_2\chi_2(p)}{p^{1/2+i\g}}\Big)^j \Big(\overline{\sum_{p\leq X^2}\sdfrac{a_1\chi_1(p)+a_2\chi_2(p)}{p^{1/2+i\g}}}\Big)^{k-j}.
        \ee
We write
     \[
     \CMcal{P}_{\mathcal L}(\g)^j
     =\sum_{\substack{n=p_1\dots p_j,\\ p_i \leq X^2}}  \frac{b_j(n)\kappa(n)}{n^{1/2+i\g}} 
     \quad \text{and}
     \quad
      \CMcal{P}_{\mathcal L}(\g)^{k-j}
     =\sum_{\substack{m=q_1\dots q_{k-j},\\ q_i \leq X^2}}  \frac{b_{k-j}(m)\kappa(m)}{m^{1/2+i\g}} 
     \]   
     Here $p_1,\dots, p_j, q_1,\dots, q_{k-j}$ denote primes that are not necessarily distinct, and the coefficients $b_r(p_{i_1}\dots p_{i_r})$ denotes the number of permutations of the primes $p_{i_1},\dots, p_{i_r}$. We also set
    \[
    \kappa(p_{i_1}\dots p_{i_r}) 
    = \bigs(a_1\chi_1(p_{i_1})+a_2\chi_2(p_{i_1})\bigs) \dots \bigs(a_1\chi_1(p_{i_r})+a_2\chi_2(p_{i_r})\bigs). 
    \] 
Using these notations and Corollary 3.2 in~\cite{log zeta}, we find that \eqref{j term linear combination} is 
         \[
        \begin{split}
        N(T)&\sum_{n} \frac{b_j(n)\kappa(n)}{\sqrt n} \frac{b_{k-j}(n)\overline{\kappa(n)}}{\sqrt n}  \\
        -\frac{T}{2\pi}&\sum_{n}\sum_{m} 
        \frac{b_j(n)\kappa(n)}{\sqrt n}
        \frac{b_{k-j}(m)\overline{\kappa(m)}}{\sqrt m}
        \bigg\{\frac{\Lambda(m/n)}{\sqrt{m/n}}+\frac{\Lambda(n/m)}{\sqrt{n/m}}\bigg\} \\
         &+O\bigg(X^{2k}\log T \log\log T\sum_n\sdfrac{b_j(n)|\kappa(n)|}{n}
         \sum_{n<m}b_{k-j}(m) |\kappa(m)| \bigg) \\
        &+O\bigg(X^{2k}\log T \log\log T
        +\sum_m\sdfrac{b_{k-j}(m) |\kappa(m)|}{m}\sum_{m<n} b_j(n) |\kappa(n)|\bigg) \\
        &+O\bigg(X^{2k}\log^2{T}\Big(\sum_n\sdfrac{b_j(n)^2|\kappa(n)|^2}{n}
        +\sum_m\sdfrac{b_{k-j}(m)^2 |\kappa(m)|^2 }{m}\Big)\bigg).
        \end{split}
        \]
In the above equation, we again have $n=p_1\dots p_j$ and $m=q_1\dots q_{k-j}$. The rest of the proof is similar to that of Lemma 5.1 in~\cite{log zeta}. The main terms in the above will match the main terms on the right-hand side of \eqref{eq:lemma 3.4 linear combination}. Similarly to the estimation of the error terms in the proof of Lemma 5.1, one can use the bound $|\kappa(p_{i_1}\dots p_{i_r})| \leq (|a_1|+|a_2|)^r$, and estimate the error terms in the above.

\end{proof}

We now compute the Fourier transform of $\Re\CMcal{P}_{\mathcal L}(\g)$.

\begin{lem}\label{fourier linear combination}
Let $\Omega=\Psi(T)^2$ where $\Psi(T)=\sum_{p\leq T} \frac1p$. For $\omega\in [0, \Omega]$,
        \be\notag
        \begin{split}
        \sum_{0 < \g \leq T}\exp\bigs( 2\pi i \omega&\Re{\CMcal{P}_{\mathcal L}(\g)}\bigs) 
        = N(T)\int_0^1
         \exp\bigs(2\pi i \omega \Re{\CMcal{P}_{\mathcal L}(\underline{\theta})}\bigs)\mathop{d\underline{\theta}} \\
        &-\frac{T}{\pi}\sum_{q\leq X^2} \sum_{1\leq\ell\leq K} \frac{\log q}{q^{\ell/2}}
        \int_0^1\exp\bigs(2\pi i\omega\Re{\CMcal{P}_{\mathcal L}(\underline{\theta})}\bigs)
        \Re( e^{2\pi i\ell\theta_q})\mathop{d\underline{\theta}} 
        +O\bigg( \frac{N(T)\omega}{2^K}\bigg).
        \end{split}
        \ee
\end{lem}

\begin{proof}
The proof is very similar to that of Lemma 5.4 in~\cite{log zeta}. In order to not be repetitive, we will only mention the important parts of the proof.

The basic idea is to use Lemma \ref{lemma 3.4 linear combination} together with the estimate 
    \be\label{eq:exponential}
    e^{ix}=\sum_{0\leq k<K}\frac{(ix)^k}{k!}+O\bigg(\frac{|x|^K}{K!}\bigg) \quad \text{for} \quad x\in\mathbb{R}.
    \ee   
Note that the above approximation is not good for all $K$, so we set our parameters as follows. Let $\Psi(T)=\sum_{p\leq T} p^{-1}$ as before and
    \be\label{Omega K}
    \Omega=\Psi(T)^2, \quad  K=2\lfloor \Psi(T)^6\rfloor   \quad
     \text{and} \quad X \leq T^{\frac{1}{16\Psi(T)^6}}.
    \ee
By \eqref{eq:exponential}, 
        \be\label{eq:exponential expansion linear combination}
        \begin{split}
        &\sum_{0 < \g \leq T} \exp\bigs(2\pi i \omega\Re{\CMcal{P}_{\mathcal L}(\g)}\bigs) \\
        =&\sum_{0\leq k< K}\frac{(2\pi i \omega)^k}{k!}
        \sum_{0 < \g \leq T}\bigs(\Re{\CMcal{P}_{\mathcal L}(\g)}\bigs)^k
        +O\bigg(\sdfrac{(2\pi \omega)^K}{K!}\sum_{0 < \g \leq T}
        \bigs|\Re{\CMcal{P}_{\mathcal L}(\g)}\bigs|^K\bigg).
        \end{split}
        \ee
 By Lemma \ref{our lemma 3.16 linear combination} and the Stirling's formula, the $O$-term is 
        \[
        \ll N(T)\frac{(2\pi\omega)^K}{K!}(cK\Psi)^{K/2} 
        \ll N(T) \omega \frac{(2\pi e)^K\omega^{K-1} }{K^K}(cK\Psi)^{K/2}
        \ll N(T)\frac{\omega}{2^{K}}.
        \]
Now we apply Lemma \ref{lemma 3.4 linear combination} to the main term on the right-hand side of \eqref{eq:exponential expansion linear combination} and write
        \be\label{eq:fourier linear combination} 
        \begin{split}
        \sum_{0\leq k< K}\frac{(2\pi i \omega)^k}{k!}\sum_{0 < \g \leq T}\bigs(\Re{\CMcal{P}_{\mathcal L}(\g)}\bigs)^k 
        &=N(T)\sum_{0\leq k< K}\frac{(2\pi i \omega)^k}{k!}\int_0^1
         \bigs( \Re{\CMcal{P}_{\mathcal L}(\underline{\theta})}\bigs)^k \mathop{d\underline{\theta}}  \\
        -\frac{T}{\pi}
         \sum_{0\leq k< K}\frac{(2\pi i \omega)^k}{k!} 
        & \sum_{q\leq X^2}  \sum_{1\leq\ell\leq k}
        \frac{\log q}{q^{\ell/2}}\int_0^1
        \bigs( \Re{\CMcal{P}_{\mathcal L}(\underline{\theta})}\bigs)^k 
        \Re( e^{2\pi i\ell\theta_q})\mathop{d{\underline{\theta}}}   \\
         &\quad +O\bigg(\sqrt T\log^3{T}\sum_{1\leq k< K}\frac{(2\pi \omega)^k}{k!}(ck)^k\bigg). \\
        &= S_1+ S_2+S_3 . 
        \end{split}
        \ee
The sum in the above error term starts from $k=1$ because for $k=0$, there is no error term in \eqref{eq:lemma 3.4 linear combination}, that is, the statement of Lemma \ref{lemma 3.4 linear combination}.
    
The estimation of $S_1$ is easy. It can be shown by \eqref{eq:exponential} and \eqref{Omega K} that
        \be\label{S1 fourier linear combination}
        S_1
        =N(T)\int_0^1 \exp\bigs(2\pi i \omega \Re{\CMcal{P}_{\mathcal L}(\underline{\theta})}\bigs) \mathop{d\underline{\theta}} 
        +O\bigg(N(T) \frac{\omega}{2^{K}}\bigg).
        \ee
$S_3$ can be easily estimated by using \eqref{Omega K}. 
        \be\label{S3 fourier linear combination}
        S_3 
        \ll \sqrt{T}\log^3{T} \sum_{1\leq k< K}\frac{(2\pi\omega)^k}{k!}(ck)^k
        \ll  \omega \, \Omega^K\sqrt{T}\log^3{T} \,  e^{cK}
        \ll    \frac{N(T)\omega}{2^K}.
        \ee      
The estimation of $S_2$ in \eqref{eq:fourier linear combination} is more tricky. First, we extend the inner sum in $S_2$ from 
   $1\leq \ell \leq k$ to $1\leq \ell \leq K.$ This can be explained as follows. By the binomial theorem,
        \[
        \int_0^1\bigs( \Re{\CMcal{P}_{\mathcal L}(\underline{\theta})}\bigs)^k e^{2\pi i\ell\theta_q}\mathop{d{\underline{\theta}}} 
        = \frac{1}{2^k} \sum_{j=0}^k \binom{k}{j} \\
       \int_0^1 \CMcal{P}_{\mathcal L}(\underline{\theta})^{j} \overline{\CMcal{P}_{\mathcal L}(\underline{\theta})}^{k-j}
        e^{2\pi i\ell\theta_q} \mathop{d{\underline{\theta}}}.
        \]
Then we use the definition of $\CMcal{P}_{\mathcal L}(\underline{\theta})$, and compare the number of primes by using  \eqref{orthogonality} to see that the above is $0$ unless $k-j=j+\ell$ for some $0\leq j\leq k.$ This implies $1\leq \ell \leq k.$ 
The same is true for the complex conjugate of this integral. Thus, the integral in $S_2$ is $0$ for $\ell> k$ and we can write
    \be\label{S2 fourier linear combination}
    S_2
    =  -\frac{T}{\pi}\sum_{q\leq X^2} \sum_{1\leq\ell\leq K}
     \frac{\log q}{q^{\ell/2}} I_\ell (q),
    \ee
where we let
    \[
    I_\ell (q)= \int_0^1
         \bigg( \sum_{0\leq k< K} 
         \frac{\bigs(2\pi i \omega \Re{\CMcal{P}_{\mathcal L}(\underline{\theta})}\bigs)^k}{k!}  \bigg)
         \Re (e^{2\pi i\ell\theta_q})\mathop{d\underline{\theta}}.
    \]    
We know that the sum over $k$ in $I_\ell(q)$ is approximated by $\exp\bigs(2\pi i\omega \Re{\CMcal{P}_{\mathcal L}(\underline{\theta})}\bigs)$. If $\ell>1$, then 
    \be\label{I_ell linear combination}
    I_\ell(q)= \int_0^1 \exp\bigs(2\pi i\omega \Re{\CMcal{P}_{\mathcal L}(\underline{\theta})}\bigs)
     \Re (e^{2\pi i\ell\theta_q}) \mathop{d\underline{\theta}}
    +O\bigg(\frac{\omega}{2^K}\bigg).
    \ee
This follows easily from \eqref{eq:exponential}, Lemma \ref{Tsang lemma 3.4 linear combination} and our choice of parameters in \eqref{Omega K}. The study of the case of $\ell=1$ is more intricate. We write   
       \be\label{I1+T1 linear combination}
       \begin{split}     
           \int_0^1 \exp\bigs(2\pi i\omega \Re{\CMcal{P}_{\mathcal L}(\underline{\theta})}\bigs)
           \Re(e^{2\pi i  \theta_q})\mathop{d\underline{\theta}} 
            =  I_1(q) +  T_1(q),
        \end{split}
         \ee   
where
    \[
    T_1(q)
     =\sum_{k\geq  K}\frac{(2\pi i \omega)^k}{k!} 
           \int_0^1\bigs( \Re{\CMcal{P}_{\mathcal L}(\underline{\theta})}\bigs)^k 
           \Re(e^{2\pi i\theta_q})\mathop{d{\underline{\theta}}}.
    \]         
We want to show that $T_1(q)$ is small. For this, it suffices to estimate 
    \[
    {T_1}^{'}(q)
       = \sum_{k\geq K}\frac{(2\pi i \omega)^k}{k!} 
           \int_0^1\bigs( \Re{\CMcal{P}_{\mathcal L}(\underline{\theta})}\bigs)^k e^{2\pi i\theta_q}\mathop{d{\underline{\theta}}}.
    \]  
Now, the integral in ${T_1}^{'}(q)$ is equal to
        \begin{align*}
        &\frac{1}{2^k} \sum_{j=0}^k \binom{k}{j}   \\
         \times& \int_0^1 \bigg(\sum_{p\leq X^2}\frac{(a_1\chi_1(p)+a_2\chi_2(p))e^{2\pi i\theta_p}}{\sqrt p}\bigg)^{j}
       \bigg(\overline{\sum_{p\leq X^2}\frac{(a_1\chi_1(p)+a_2\chi_2(p))e^{2\pi i\theta_p}}{\sqrt p}}\bigg)^{k-j}e^{2\pi i\theta_q}
        \mathop{d{\underline{\theta}}}.
        \end{align*}
By \eqref{orthogonality}, the integral on the right-hand side is $0$ unless $j+1=k-j,$ that is, $j=\frac{k-1}{2}.$ In that case, we have $q_1\dots q_{(k+1)/2}=p_1\dots p_{(k-1)/2}\,q\,$ for some primes $p_1,\dots,p_{(k-1)/2}, q_1,\dots, q_{(k+1)/2}.$
Then the above is 
    \[
    \frac{1}{2^k\sqrt q}
    \binom{k}{\frac{k-1}{2}} 
   \sum_{p_1,\dots,p_{(k-1)/2}  \leq X^2}\sdfrac{\varkappa_{(k-1)/2}(p_1\dots p_{(k-1)/2})
   \overline{\varkappa_{(k+1)/2}( p_1\dots p_{(k-1)/2}q)}}{p_1\dots p_{(k-1)/2}},
     \]
   where we set
   \[
   \varkappa_r(p_{i_1}\dots p_{i_r})
   = b_r(p_{i_1}\dots p_{i_r}) (a_1\chi_1(p_{i_1})+a_2\chi_2(p_{i_1}))\dots (a_1\chi_1(p_{i_r})+a_2\chi_2(p_{i_r})),
   \]
   and $b_r(p_{i_1}\dots p_{i_r})$ denotes the number of permutations of primes $p_{i_1},\dots, p_{i_r}$. In particular, $b_{(k-1)/2}( p_1\dots p_{(k-1)/2}) \leq ((k-1)/2)!$. Note that  
     \[
     b_{(k+1)/2}( p_1\dots p_{(k-1)/2}q)
     \leq \frac{(k+1)b_{(k-1)/2}( p_1\dots p_{(k-1)/2})}{2}, 
     \]
 and 
    \[
    \bigs|(a_1\chi_1(p_1)+a_2\chi_2(p_1))\dots (a_1\chi_1(p_{\frac{k-1}{2}})+a_2\chi_2(p_{\frac{k-1}{2}}))\bigs|
     \leq (|a_1|+|a_2|)^{\frac{k-1}{2}}.
    \]
    Using these three bounds, we obtain
    \begin{align*}
     \sum_{p_1,\dots,p_{(k-1)/2}  \leq X^2}& \sdfrac{\varkappa_{(k-1)/2}(p_1\dots p_{(k-1)/2})
   \overline{\varkappa_{(k+1)/2}( p_1\dots p_{(k-1)/2}q)}}{p_1\dots p_{(k-1)/2}} \\
    & \leq k (|a_1|+|a_2|)^{\frac{k-1}{2}} \bigs((k-1)/2\bigs)! \sum_{p_1,\dots,p_{(k-1)/2}  \leq X^2}
    \frac{|\varkappa_{(k-1)/2}(p_1\dots p_{(k-1)/2})|}{p_1\dots p_{(k-1)/2}} \\
    &\leq k(|a_1|+|a_2|)^{k} \bigs((k-1)/2\bigs)! \Big(\sum_{p\leq X^2}\frac{1}{p}\Big)^{(k-1)/2}
    \leq (ck\Psi)^{\frac{k-1}{2}}.
    \end{align*}
    Thus  
     \[
      \int_0^1\bigs( \Re {\CMcal{P}_{\mathcal L}(\underline{\theta})}\bigs)^k e^{2\pi i\theta_q}\mathop{d{\underline{\theta}}}
      \ll \frac{k}{\sqrt q} (ck\Psi)^{\frac{k-1}{2}},
     \]
   and as a consequence,
        \begin{align*}
        {T_1}^{'}(q)
        \ll &\,  \frac1{\sqrt q} \sum_{k\geq K} k \frac{(2\pi \omega )^k}{k!}(ck\Psi)^{\frac{k-1}{2}} 
        \ll \,    \frac{ \omega }{\sqrt q} \sum_{k\geq K} \frac{(2\pi c\, \omega  \sqrt{k\Psi})^{k-1}}{(k-1)!}  
        \ll  \frac{\omega}{2^K \sqrt q}.
        \end{align*}
Similarly, it can be seen that
\[
 \sum_{k\geq K}\frac{(2\pi i \omega)^k}{k!} 
       \int_0^1\bigs( \Re {\CMcal{P}_{\mathcal L}(\underline{\theta})}\bigs)^k e^{-2\pi i\theta_q}\mathop{d{\underline{\theta}}}
       \ll  \omega/(2^K\sqrt q).
\]
Hence $T_1(q) \ll \omega/(2^K\sqrt q).$ Then by \eqref{I1+T1 linear combination},
    \[
     I_1(q) 
     =\int_0^1   \exp\bigs(2\pi i\omega \Re {\CMcal{P}_{\mathcal L}(\underline{\theta})}\bigs)
     \Re(e^{2\pi i  \theta_q})\mathop{d\underline{\theta}} 
    +O\bigg(\frac{\omega}{2^K \sqrt q}\bigg).
     \]
Combining this with \eqref{I_ell linear combination} in  \eqref{S2 fourier linear combination}, we find that
    \[
     S_2  
     =-\frac{T}{\pi} \sum_{q\leq X^2}   \sum_{ 1\leq\ell\leq K}
     \frac{\log q}{q^{\ell/2}} \; 
        \int_0^1   \exp{  \bigs(2\pi i\omega \Re{\CMcal{P}_{\mathcal L}(\underline{\theta})}\bigs)}
        \Re(e^{2\pi i \ell \theta_q}) \mathop{d\underline{\theta}}  
        +O\bigg(\frac{T\omega}{2^K} \sum_{q\leq X^2} \frac{\log q}{q}\bigg).
      \]
The sum in the error term is $\ll \log T,$ so 
     \[
      S_2 
       =-\frac{T}{\pi}  \sum_{q\leq X^2}   \sum_{ 1\leq\ell\leq K}
       \frac{\log q}{q^{\ell/2}} \; 
        \int_0^1 \exp\bigs(2\pi i\omega \Re {\CMcal{P}_{\mathcal L}(\underline{\theta})}\bigs)
        \Re(e^{2\pi i \ell \theta_q})\mathop{d\underline{\theta}} 
        +O\bigg(\frac{N(T)\omega}{2^K}\bigg) .
      \]
We substitute this estimate and the results of \eqref{S1 fourier linear combination} and \eqref{S3 fourier linear combination} into \eqref{eq:fourier linear combination}. The proof is then complete.

 \end{proof}

%%%%%%%%%%%%%%%%%%%%%%%%%%%%%%%%%  

We now prove that the distribution function of $\Re\CMcal{P}_{\mathcal L}(\g)$ can be expressed in terms of the distribution function of a suitable Gaussian random variable.

\begin{prop}\label{distr of Re P linear combination}
    Let $\displaystyle X = T^{\frac{1}{16(\log\log T)^6}}$. Then for real numbers $A<B$,
    \begin{align*}
    \frac{1}{N(T)}\sum_{0<\g\leq T}  \mathbbm{1}_{[A, B]}
    \bigs(\Re \CMcal{P}_{\mathcal L}(\g)\bigs) 
    =  \frac{1}{\sqrt{2\pi}}
    \int_{A/\sqrt{(\tfrac{a_1^2+a_2^2}{2})\log\log T}}^{B/\sqrt{(\tfrac{a_1^2+a_2^2}{2})\log\log T}}
       e^{-\tfrac{x^2}{2}}\mathop{dx} 
    +O\bigg(\sdfrac{\log\log\log T}{\log\log T}\bigg).
    \end{align*}
\end{prop}

\begin{proof}
The Beurling-Selberg approximation~\cite[p. 213]{Selberg sieves} states that
    \be\label{eq:F}
    \mathbbm{1}_{[A, B]}(x)
    =\frac12F_\Omega(x-A)-\frac12F_\Omega(x-B)
    +O\bigg(\sdfrac{\sin^2(\pi\Omega(x-A))}{(\pi\Omega(x-A))^2}\bigg)
    +O\bigg(\sdfrac{\sin^2(\pi\Omega(x-B))}{(\pi\Omega(x-B))^2}\bigg),
    \ee
where 
    \[
    F_{\Omega}(x)
    =\Im  \int_0^\Omega G\Big(\frac \omega \Omega\Big)\exp{(2\pi ix\omega)}\frac{\mathop{d\omega}}{\omega}
    \]
with
    \[
    G(u)= \frac{2u}{\pi}+2u(1-u)\cot{(\pi u)} \quad \text{for} \quad u\in[0, 1]. 
    \]   
To recall, we repeat the choice of parameters $\Omega, K$ and $X$ in \eqref{Omega K},
\[
  \Omega=\Psi(T)^2, \quad  K=2\lfloor \Psi(T)^6\rfloor   \quad
     \text{and} \quad X \leq T^{\frac{1}{16\Psi(T)^6}}
\]
and set  $\Psi_{\mathcal{L}}=(a_1^2+a_2^2)\Psi$ where $\Psi=\sum_{p\leq X^2} p^{-1}$. Our plan is to show that
        \be\label{eq:sum F linear combination}
        \,\mathcal F
        :=\sum_{0 < \g \leq T}F_\Omega\big(\Re\CMcal{P}_{\mathcal L}(\g)-A\big) 
        =\frac{N(T)}{\sqrt{2\pi}}\int_{-\infty}^\infty\sgn{\bigg(x-\frac{A}{\sqrt{\frac12\Psi_{\mathcal L}}}\bigg)}e^{-\tfrac{x^2}{2}}\mathop{dx}
        +O\bigg(\frac{N(T)}{{\Psi_{\mathcal{L}}}^2}\bigg),
        \ee
  and 
        \be\label{eq:sum sin linear combination}
        \mathcal{E}
        :=\sum_{0 < \g \leq T}\frac{\sin^2\big(\pi \Omega(\Re{\CMcal{P}_{\mathcal L}(\g)}-A)\big)}
        {\big(\pi \Omega(\Re{\CMcal{P}_{\mathcal L}(\g)}-A)\big)^2} 
        =O\bigg( \frac{N(T)}{{\Psi_{\mathcal{L}}}^2}\bigg).
        \ee
These results remain valid when $A$ is replaced with $B$. Then by \eqref{eq:F}, it will follow that
        \[
         \sum_{0<\g\leq T}\mathbbm{1}_{[A, B]}
         \bigs(\Re\CMcal{P}_{\mathcal L}(\g)\bigs)
          =\frac{N(T)}{\sqrt{2\pi}}
          \int_{A/\sqrt{\frac12\Psi_{\mathcal L}}}^{B/\sqrt{\frac12\Psi_{\mathcal L}}} 
     e^{-\tfrac{x^2}{2}}\mathop{dx}
          +O\bigg(\frac{N(T)}{{\Psi_{\mathcal{L}}}^2}\bigg).
         \]
Here, for our choice of $X$, 
        \[
        \Psi_{\mathcal L}=(a_1^2+a_2^2)\Psi+O(1)=(a_1^2+a_2^2)\log\log T+O\bigs(\log\log\log T\bigs). 
        \] 
Thus the integral in the previous statement equals
       \[
       \frac{N(T)}{\sqrt{2\pi}}
       \int_{A/\sqrt{(\frac{a_1^2+a_2^2}{2})\log\log T}}^{B/\sqrt{(\frac{a_1^2+a_2^2}{2})\log\log T}}
       e^{-\tfrac{x^2}{2}}\mathop{dx}+O\bigg(N(T)\sdfrac{\log\log\log T}{\log\log T}\bigg). 
       \]
 Since 
       $\displaystyle
       \tfrac{N(T)}{{\Psi_{\mathcal{L}}}^2}
       \ll N(T)\sdfrac{\log\log\log T}{\log\log T},
       $
the proof of the proposition will then be complete. 
       
We start with the definition in \eqref{eq:F} and write
        \[
        \mathcal F
        =\Im \int_0^ \Omega G\Big(\frac{\omega}{\Omega}\Big)e^{-2\pi i A\omega}\sum_{0 < \g \leq T}
        \exp\bigs(2\pi i\omega\Re\CMcal{P}_{\mathcal L}(\g)\bigs)\frac{\mathop{d\omega}}{\omega}. 
        \]
 By Lemma \ref{fourier linear combination}, we write
        \be\label{F term linear combination}
        \begin{split}
        \mathcal F
        =&\, N(T)  \Im  \int_0^\Omega G\Big(\frac{\omega}{\Omega}\Big)e^{-2\pi i A\omega}
        \int_0^1\exp\bigs(2\pi i\omega \Re \CMcal{P}_{\mathcal L}(\underline{\theta})\bigs)\mathop{d\underline{\theta}}
        \frac{\mathop{d\omega}}{\omega}  \\
        &-\sdfrac{T}{\pi}\sum_{\substack{q\leq X^2, \\ 1\leq\ell \leq K}}
        \sdfrac{\log q}{q^{\ell/2}}
         \Im  \int_0^{\Omega}G\Big(\sdfrac{\omega}{\Omega}\Big)e^{-2\pi iA\omega}
        \int_0^1\exp\bigs(2\pi i\omega\Re{\CMcal{P}_{\mathcal L}(\underline{\theta})}\bigs) 
        \Re(e^{2\pi i\ell\theta_q})\mathop{d\underline{\theta}}\frac{\mathop{d\omega}}{\omega} \\
         &+O\bigg(\sdfrac{N(T)}{2^K}\int_0^\Omega \Big| G\Big(\frac{\omega}{\Omega}\Big)e^{-2\pi iA\omega}\Big| \mathop{d\omega} \bigg) 
        =\, \mathcal F_1+\mathcal F_2+O\bigg(N(T)  \frac{\Omega}{2^K} \bigg).
        \end{split}    
        \ee
Here, we used the boundedness of the function $G$ on $[0, 1]$ to estimate the $O$-term. To prove \eqref{eq:sum sin linear combination}, we start with the identity 
        \[
        \frac{\sin^2(\pi \Omega x)}{(\pi \Omega x)^2}
        =\frac{2}{\Omega^2}\int_0^{\Omega}(\Omega-\omega)
        \cos(2\pi x\omega)\mathop{d\omega},
        \]
and write
        \[
        \mathcal E
        =\frac{2}{\Omega^2}\int_0^{\Omega}(\Omega-\omega)
        \Re{\sum_{0 < \g \leq T}\exp\bigs(2\pi i\omega(\Re\CMcal{P}_{\mathcal L}(\g)-A)\bigs)}\mathop{d\omega}.
        \]
 Again by Lemma \ref{fourier linear combination}, 
        \be\label{sin term linear combination}
        \begin{split}
        \mathcal E
        \ll & \,  \frac{N(T)}{\Omega^2}\int_0^\Omega (\Omega-\omega)\bigg|    \int_0^1
        \exp\bigs(2\pi i\omega\Re{\CMcal{P}_{\mathcal L}(\underline{\theta})}\bigs)
        \mathop{d\underline{\theta}}\bigg|\mathop{d\omega}  \\
        & +\frac{T}{\Omega^2}\sum_{q\leq X^2} \sum_{1\leq\ell \leq K}
        \frac{\log q}{q^{\ell/2}}\int_0^\Omega (\Omega-\omega)\bigg|
         \int_0^1\exp\bigs(2\pi i\omega\Re{\CMcal{P}_{\mathcal L}(\underline{\theta})}\bigs) 
         \Re(e^{2\pi i\ell\theta_q})\mathop{d\underline{\theta}}
         \bigg|\mathop{d\omega}   \\
        &+ \, N(T)\frac{\Omega}{2^K} 
        = \, {\mathcal E}_1+{\mathcal E}_2+O\bigg(N(T)\frac{\Omega}{2^K}\bigg) .
        \end{split}   
        \ee
Note that
    \begin{align*}
    2 \int_0^1 \exp\bigs(2\pi i\omega\Re&{\CMcal{P}_{\mathcal L}(\underline{\theta})}\bigs) 
      \Re(e^{2\pi i\ell\theta_q})\mathop{d\underline{\theta}}       \\
    =   &\int_0^1\exp\bigs(2\pi i\omega\Re{\CMcal{P}_{\mathcal L}(\underline{\theta})}\bigs) 
     e^{2\pi i\ell\theta_q}\mathop{d\underline{\theta}} 
     +\int_0^1\exp\bigs(2\pi i\omega\Re{\CMcal{P}_{\mathcal L}(\underline{\theta})}\bigs) 
     e^{-2\pi i\ell\theta_q}\mathop{d\underline{\theta}}.
    \end{align*}
Observe that all of $\mathcal F_1,  \mathcal F_2, \mathcal S_1$ and $\mathcal S_2$ include two types of integrals over $\underline{\theta}$ and one integral is a special case of the other for $\ell=0$. We will now compute these integrals. Define
    \[
     {\mathcal J}_\ell
     :=  \int_0^1\exp\bigs(2\pi i\omega\Re{\CMcal{P}_{\mathcal L}(\underline{\theta})}\bigs)
      e^{2\pi i\ell\theta_q}\mathop{d\underline{\theta}}.
    \]
In the definition of $\Re\CMcal{P}_{\mathcal L}(\underline{\theta})$, let
   \[
   \nu_p=|a_1\chi_1(p)+a_2\chi_2(p)| \quad \text{and} \quad 
   a_1\chi_1(p)+a_2\chi_2(p)=  \nu_p e^{2\pi i\b_p} \quad \text{for} \quad  0\leq \b_p <1.
   \] 
If $\nu_p=0,$ then we set $\b_p=0$. By using the independence of $(\theta_p)$, we see that
        \begin{align*}
        {\mathcal J}_\ell
        =\int_0^1 \exp\bigg(2\pi i\Big(\omega\sdfrac{\nu_q\cos(2\pi(\theta_q+\b_q))}{\sqrt q}+\ell\theta_q\Big)&\bigg)\mathop{d{\theta_q}} \\
        \cdot \prod_{\substack{p\leq X^2\\ p\neq q}} &\int_0^1 
        \exp\Big(2\pi i\omega\sdfrac{\nu_p\cos(2\pi(\theta_p+\b_p))}{\sqrt p}\Big)\mathop{d{\theta_p}}.
        \end{align*}      
Recall that for $\ell \in \mathbb{Z}^+$, the Bessel function of integer order $\ell$ is given by
    \be\label{Jell integral}
    J_\ell(z)
    =(-i)^\ell\int_0^1 \exp\bigg(2\pi i \Big(\sdfrac{z\cos(2\pi\theta)}{2\pi}+\ell\theta\Big)\bigg)\mathop{d\theta}.
    \ee          
After a change of variable, the integral in ${\mathcal J}_\ell$ that corresponds to a prime $p\neq q$ is seen to equal
     \[
    J_0\Big(\frac{2\pi\omega \nu_p}{\sqrt p}\Big).
     \]
Via the change of variable $\theta_q \to \theta_q-\b_q$, the integral that corresponds to the prime $q$ equals
       \[
        \int_{-\b_q}^{1-\b_q}
         \exp\bigg(2\pi i\Big(\omega \nu_q\frac{\cos(2\pi\theta_q)}{\sqrt q}+\ell\big(\theta_q-\b_q\big)\Big)\bigg)
        \mathop{d{\theta_q}}.
        \]
Since the integrand has period $1$, it follows from \eqref{Jell integral} that this is
        \[
         (ie^{-2\pi i \b_q})^\ell J_{\ell}\Big(\frac{2\pi \nu_q \omega}{\sqrt q}\Big).
        \]     
 By multiplying the expressions for the integrals, we obtain for $\ell\in \mathbb{Z}^+$
        \be\label{J linear combination}
        {\mathcal J}_\ell
        =(ie^{-2\pi i \b_q})^\ell J_{\ell}\Big(\frac{2\pi  \nu_q \omega}{\sqrt q}\Big)
        \prod_{\substack{p\leq X^2,\\p\neq q}}J_0\Big(\frac{2\pi \nu_p \omega}{\sqrt p}\Big).
        \ee       
Since $J_{\ell}(z)=(-1)^\ell J_{\ell}(z)$, for $\ell\in \mathbb{Z}^+$ we also have
        \[
        {\mathcal J}_{-\ell}
        =(-ie^{-2\pi i \b_q})^\ell J_{\ell}\Big(\frac{2\pi  \nu_q \omega}{\sqrt q}\Big)
        \prod_{\substack{p\leq X^2,\\p\neq q}}J_0\Big(\frac{2\pi \nu_p \omega}{\sqrt p}\Big).
        \]
By using the bound $\nu_p \leq |a_1|+|a_2|$ and the argument in \cite[p. 34]{log zeta}, we find that if $\ell \in \mathbb{Z}$ and $\omega \in[0, \Omega]$, then
       \be\label{product Bessels 2 linear combination}
        {\mathcal J}_{\ell}
       \ll \frac{((|a_1|+|a_2|)\pi\omega)^{|\ell|}}{|\ell| !q^{|\ell|/2}}  e^{-c\Psi\omega^2+\sqrt{2}\pi(|a_1|+|a_2|)\omega}
       \ll \frac{((|a_1|+|a_2|)\pi\omega)^{|\ell|}}{|\ell| !q^{|\ell|/2}}  e^{-c\Psi\omega^2}.
        \ee        
Thus
     \[
     {\mathcal E}_1      
     \ll   \frac{N(T)}{\Omega^2}\int_0^\Omega
     (\Omega-\omega)e^{-c\Psi\omega^2}\mathop{d\omega} 
     \leq \frac{N(T)}{\Omega }\int_0^\infty e^{-c\Psi\omega^2}\mathop{d\omega} 
     \ll  \frac{N(T)}{\Omega}.
      \]
We use this bound to estimate ${\mathcal E}_2$ in \eqref{sin term linear combination} as
        \begin{align*}
        {\mathcal E}_2
        \ll &\,  \frac{T}{\Omega^2}\sum_{1\leq\ell \leq K} \frac{(\pi(|a_1|+|a_2|))^\ell}{\ell!}
        \sum_{q\leq X^2}  \frac{\log q}{q^{\ell}}\int_0^\Omega
        (\Omega-\omega)\omega^\ell e^{-c\Psi\omega^2}\mathop{d\omega} \\
        \ll &\, \frac{T}{\Omega} \sum_{q\leq X^2}  \frac{\log q}{q}
       \int_0^\Omega   \bigg( \sum_{1\leq\ell \leq K} \frac{(\pi(|a_1|+|a_2|)\omega)^\ell}{\ell!}\bigg)
       e^{-c\Psi\omega^2}\mathop{d\omega} \\
        \ll &\, \frac{T}{\Omega}\sum_{q\leq X^2}  \frac{\log q}{q}
        \int_0^\Omega e^{-c\Psi\omega^2+(|a_1|+|a_2|)\pi\omega}\mathop{d\omega}
        \ll \frac{T\log X}{\Omega \sqrt \Psi}\ll  \frac{N(T)}{\Omega \sqrt \Psi}.
        \end{align*} 
Combining our estimates for  $\mathcal E_1$ and $\mathcal E_2$ in \eqref{sin term linear combination}, we obtain
       \[
        \mathcal E
        \ll \frac{N(T)}{\Omega \sqrt \Psi} + \frac{N(T) \Omega}{2^K}.
        \]
By \eqref{Omega K}, this is $\ll N(T)/{\Psi_{\mathcal{L}}}^2$, so we obtain \eqref{eq:sum sin linear combination}. We proceed with the proof of \eqref{eq:sum F linear combination}. By \eqref{J linear combination},
        \[
        \mathcal F_1
        = N(T)\Im\int_0^{\Omega}G\Big(\frac{\omega}{\Omega}\Big)e^{-2\pi iA\omega}\prod_{p\leq X^2}J_0\Big(\frac{2\pi \nu_p \omega}{\sqrt p}\Big)\frac{\mathop{d\omega}}{\omega}.
        \]
Similarly to a result of Tsang's (see~\cite[pp. 34--35]{Tsang}), one can prove that      
         \[
        {\mathcal F}_1
        =\frac{N(T)}{\sqrt{2\pi}} \int_{-\infty}^\infty 
        \sgn\bigg(x-\sdfrac{A}{\sqrt{\frac12\Psi_{\mathcal L}}}\bigg)e^{-\tfrac{x^2}{2}}\mathop{dx}
        +\, O\bigg(\frac{N(T)}{{\Psi_{\mathcal{L}}}^2}\bigg).
        \]
We see that  
        \begin{align*}
        {\mathcal F}_2
         =&-\frac{T}{2\pi}\sum_{q\leq X^2}\sum_{1\leq\ell \leq K} \frac{\log q}{q^{\ell/2}}
         \Im \int_0^\Omega G\Big(\frac{\omega}{\Omega}\Big)e^{-2\pi iA\omega}
          \bigs({\mathcal J}_\ell +{\mathcal J}_{-\ell} \bigs) \frac{\mathop{d\omega}}{\omega}.
         \end{align*} 
Now by \eqref{product Bessels 2 linear combination}, 
        \[
        \begin{split}
        &- \frac{T}{2\pi} \sum_{q\leq X^2}\sum_{1\leq\ell \leq K} 
       \frac{\log q}{q^{\ell/2}}
       \;  \Im \int_0^\Omega G\Big(\frac{\omega}{\Omega}\Big)e^{-2\pi iA\omega}\,
       {\mathcal J}_\ell \frac{\mathop{d\omega}}{\omega} \\
       \ll& \, \, T \sum_{\substack{q\leq X^2,\\ 1\leq\ell\leq K}}
        \frac{(\pi(|a_1|+|a_2|)) ^\ell \log q}{\ell! q^{\ell}}
         \int_0^\Omega G\Big(\frac{\omega}{\Omega}\Big)
        \omega^{\ell-1} e^{-c\Psi\omega^2} \mathop{d\omega}.
        \end{split}
        \]      
Since $G$ is bounded on $[0, 1],$ the above is
        \begin{align*}  
        & \ll \, T   \sum_{q\leq X^2}\frac{\log q}{q} 
        \sum_{1\leq\ell \leq K}  \frac{(\pi(|a_1|+|a_2|))^\ell}{\ell!}
         \int_0^\Omega \omega^{\ell-1}
         e^{-c\Psi\omega^2}\mathop{d\omega} \\
      &\ll    T   \sum_{q\leq X^2}\frac{\log q}{q}  
        \int_0^\Omega \sum_{1 \leq \ell \leq K} 
        \frac{(\pi(|a_1|+|a_2|) \omega)^{\ell-1} }{(\ell-1)!} e^{-c\Psi\omega^2}\mathop{d\omega}  \\
      & \ll   T   \sum_{q\leq X^2}\frac{\log q}{q}  
      \int_0^\Omega  e^{-c\Psi\omega^2+\pi(|a_1|+|a_2|)\omega}\mathop{d\omega}.   
      \end{align*}  
The integrand is bounded and the sum over $q$ is $\ll \log X,$ so by \eqref{Omega K}, the above is
        \[
         \ll T\, \Omega \log X
         \ll \frac{N(T)}{(\log\log T)^4}.
        \]     
Similarly, 
    \[
    - \frac{T}{2\pi} 
        \sum_{q\leq X^2}\sum_{1\leq\ell \leq K}  \frac{\log q}{q^{\ell/2}}
       \;  \Im \int_0^\Omega G\Big(\frac{\omega}{\Omega}\Big)e^{-2\pi iA\omega}\,
       {\mathcal J}_{-\ell} \frac{\mathop{d\omega}}{\omega} 
    \ll  \frac{N(T)}{(\log\log T)^4}.
    \]
Combining our estimates for ${\mathcal F}_1$ and  ${\mathcal F}_2$, we obtain \eqref{eq:sum F linear combination}.
     \end{proof}

  The proof of Theorem \ref{distr of linear combination} can now be easily completed  by following the argument in Section 5.5 of~\cite{log zeta} and by using Propositions \ref{linear combination error} and \ref{distr of Re P linear combination}. We thus contend with some remarks. First, one defines a remainder function corresponding to the situation where the Dirichlet polynomial $\CMcal{P}_{\mathcal{L}}(X,\g)$ has length $X^2$. Let
    \[
    r(X,\g)=\mathcal{L}(\r)-\Re\CMcal{P}_{\mathcal{L}}(X, \g),
    \]
where $\displaystyle X=T^{\frac{1}{16(\log\log T)^6}}$ and $\CMcal{P}_{\mathcal{L}}(X,\g)=\sum_{p\leq X^2} \frac{a_1\chi_1(p)+a_2\chi_2(p)}{p^{1/2+i\g}}$. 

For $\displaystyle Y= T^{\tfrac{1}{16k}}$ with $k=\lfloor \log{\Psi(T)}\rfloor$, one similarly defines another remainder $r(Y,\g)$ corresponding to the real part of a polynomial with length $Y^2$, that is, $\Re \CMcal{P}_{\mathcal{L}}(Y, \g).$ Note that due to the choice of $Y$, one can use Proposition \ref{linear combination error} and bound moments of $r(Y,\g).$ Since  $r(Y,\g)$ and $r(X, \g)$ are sufficiently close to each other, this can be used to bound moments of $r(X, \g).$ Via Chebyshev's inequality, one can then show that the distribution function of $\mathcal{L}(\r)$ is close to that of $\Re \CMcal{P}_{\mathcal{L}}(X, \g).$ Finally, one uses Proposition \ref{distr of Re P linear combination}, which states that the distribution function of $\Re\CMcal{P}_{\mathcal{L}}(X, \g)$ is approximately Gaussian. A careful analysis provides the rate of convergence to the Gaussian distribution function as in the statement of Theorem \ref{distr of linear combination}.

\section{Proof of Theorem \ref{independence}}

We use an argument similar to the one in~\cite{Hsu Wong}. First, let us recall some concepts and results from probability.

 \begin{defn*}
A vector $(X_1, \dots, X_n)$ of real-valued random variables is said to have an $n$-variate Gaussian distribution if for all vectors $\bm{a}\in\mathbb{R}^n,$ the scalar product $\bm{a}\cdot \bm{X}$ has a Gaussian distribution. 
 \end{defn*}

% See ~\cite[pp. 117--118]{GS} for the definition.
 
 \begin{prop}
    Let $X_1, \dots, X_n$ be random variables that have an $n$-variate Gaussian distribution. Then $X_1, \dots, X_n$ are independent if and only if $\var(X_j+X_k)=\var(X_j)+\var(X_k)$ for all $1\leq j< k \leq n.$
 \end{prop}
 
 \begin{proof}\let\qed\relax
 See ~\cite[p. 101]{GW}.
 \end{proof}

 \begin{cor}\label{indep random polyl}
Let $\chi_1, \dots, \chi_n$ be distinct primitive Dirichlet characters and $\Psi=\sum_{p\leq X^2} p^{-1}$. In the notation of the above proposition, let 
     \[
     X_j=\frac{1}{\sqrt{\Psi/2}}\Re \CMcal{P}_{\chi_j}(\underline{\theta})
     =\frac{1}{\sqrt{\Psi/2}} \Re \sum_{p\leq X^2}\frac{\chi_j(p) e^{2\pi i\theta_p}}{\sqrt p} \quad \text{for} \quad 1\leq j \leq n.
     \]
Then $X_1, \dots, X_n$ are independent. 
 \end{cor}
 
 \begin{proof}
By using a generalization of \eqref{variance linear combination} and the central limit theorem from probability, we see that $\bm{a}\cdot \bm{X}$ has a Gaussian distribution for any real vector $\bm{a}\in \mathbb{R}^n$. Thus $(X_1, \dots, X_n)$ has an $n$-variate Gaussian distribution. The claim then follows from the above proposition. 
 \end{proof}

 \begin{proof}[Proof of Theorem \ref{independence}]\let\qed\relax
For convenience, we again let
     \[
     X_j=\frac{1}{\sqrt{\Psi/2}} \Re \CMcal{P}_{\chi_j}(\underline{\theta})
    \quad \text{for} \quad 1\leq j \leq n.
     \]
One can prove an analogue of Theorem 1.1 in~\cite{LLM} by using Lemma \ref{fourier linear combination} to prove that
    \begin{align*}
    \lim_{T\to\infty } \frac{1}{N(T)}
    \#\bigg\{0<\g \leq T&: \frac{\log|L(\r, \chi_j)|}{\sqrt {\tfrac12 \log\log T}} \in [A_j, B_j] \,\, \text{ for each } \, 1\leq j \leq n \bigg\}  \\
    =\lim_{T\to\infty } 
    &\mathbb{P} \bigs( X_j\in [A_j, B_j] \,\, \text{ for each } \, 1\leq j \leq n \bigs).
    \end{align*}
Corollary \ref{indep random polyl} implies that the right-hand side equals
    \[
    \lim_{T\to\infty}\prod_{j=1}^n 
    \mathbb{P} \big( X_j\in [A_j, B_j] \big).
    \]
Also by Theorem \ref{distr of linear combination} with $a_1=1$ and $a_2=0$, for each $1 \leq j \leq n$
    \[
    \lim_{T\to\infty }\mathbb{P} \big( X_j\in [A_j, B_j] \big)
    =\lim_{T\to\infty } \frac{1}{N(T)}
    \#\bigg\{0<\g \leq T: \frac{\log|L(\r, \chi_j)|}{\sqrt {\tfrac12 \log\log T}} \in [A_j, B_j]  \bigg\}.
    \]
By combining these two statements, we complete the proof.
 \end{proof}

%%%%%%%%%%%%%%%%%%%%%%%%%%%%%%%%%%%%%%%%%%%%% BIBLIOGRAPHY

\end{document}